\documentclass[a4paper]{amsart}
\usepackage[utf8x,utf8]{inputenc}
\usepackage{amsfonts,amsmath,amssymb,amsthm}
\usepackage{amscd}
\usepackage{mathrsfs}
\usepackage{ucs,bbm} 
\usepackage{colonequals}
\usepackage[all]{xy}
\usepackage{verbatim}
\usepackage[backref=page]{hyperref}

\hypersetup{
     colorlinks   = true,
     citecolor    = red,
     linkcolor = blue,
     urlcolor = purple
}

\usepackage{enumerate}
\usepackage{xcolor}

\mathchardef\dash="2D

\newcommand{\an}{\mathrm{an}}

\newcommand{\rig}{\mathrm{rig}}
\newcommand{\eff}{\mathrm{eff}}
\newcommand{\dR}{\mathrm{dR}}

\newcommand{\Q}{\mathbb{Q}}
\newcommand{\QQ}{\mathbb{Q}}
\newcommand{\CC}{\mathbb{C}}

\newcommand{\C}{\mathbb{C}}
\newcommand{\Z}{\mathbb{Z}}

\newcommand{\GG}{\mathbb{G}}

\newcommand{\FOg}{\mathbf{FOg}}
\newcommand{\CNM}{\mathbf{CNM}}
\newcommand{\DM}{\mathbf{DM}} 
\newcommand{\Og}{\mathbf{Og}}

\newcommand{\Mod}{\mathrm{Mod}}

\newcommand{\calO}{\mathcal{O}}

\DeclareMathOperator{\End}{End}
\DeclareMathOperator{\Hom}{Hom}

\DeclareMathOperator{\Spec}{Spec}

\DeclareMathOperator{\Pic}{Pic}

\DeclareMathOperator{\D}{D}

\DeclareMathOperator{\gr}{gr}

\DeclareMathOperator{\Div}{Div}

\DeclareMathOperator*{\colim}{colim}

\newcommand{\et}{{\rm \acute{e}t}}

\newcommand{\isomto}{\overset{\sim}{\rightarrow}}

\usepackage{fancyhdr}

\title{The filtered Ogus realisation of motives}
\author{Bruno Chiarellotto}
\address{Universit\`a di Padova, Dipartimento di Matematica ``Tullio Levi-Civita''}
\email{chiarbru@math.unipd.it}
\author{Christopher Lazda}
\address{Korteweg--de Vries Institute for Mathematics, Universiteit van Amsterdam}
\email{c.d.lazda@uva.nl}
\author{Nicola Mazzari}
\address{Universit\'e de Bordeaux, IMB}
\email{nicola.mazzari@math.u-bordeaux.fr}

\newtheorem{exo}{Exercise}[section]
\newtheorem{thr}[exo]{Theorem}
\newtheorem{lmm}[exo]{Lemma}
\newtheorem{prp}[exo]{Proposition}
\newtheorem{crl}[exo]{Corollary}

\theoremstyle{definition}\newtheorem{dfn}[exo]{Definition}
\theoremstyle{remark}\newtheorem{rmk}[exo]{Remark}
\theoremstyle{remark}
\theoremstyle{definition}

\newcommand{\bu}{\bullet}
\newcommand{\isomfrom}{\overset{\sim}{\leftarrow}}
\newcommand{\ldR}{\log\text{-}\mathrm{dR}}
\newcommand{\cur}[1]{\mathcal{#1}}

\newcommand{\norm}[1]{\left\vert#1\right\vert}
\newcommand{\spec}[1]{\mathrm{Spec}\left(#1\right)}

\begin{document}


\begin{abstract}
    We construct the (filtered) Ogus realisation of Voevodsky motives over a number field $K$. This realisation extends the functor defined on $1$-motives by Andreatta, Barbieri-Viale and Bertapelle. As an illustration we note that the analogue of the Tate conjecture holds for K3 surfaces.
\end{abstract}
\maketitle


\section{Introduction}
Let $K$  be a number field, and $G_K$ its absolute Galois group.  For $X$ a $K$-scheme,  let $H_\ell^i(X)=H_\et^i(X_{\overline K},\Q_\ell)$ be the $\ell$-adic cohomology of $X$. If $X$ is projective and smooth over $K$, the Tate conjecture for divisors predicts that the $\ell$-adic cycle class map
\[
	c_{1,\ell}:\Pic(X)\otimes \Q_\ell\rightarrow H_\ell^2(X)(1)^{G_K}=\Hom_{\Q_\ell[G_K]}(\Q_\ell,H_\ell^2(X)(1))
\]
is surjective. This conjecture is known for abelian varieties thanks to Faltings \cite{Fal83}, and is equivalent to the fact that
\[
	\Hom(A,B)\otimes \Q_\ell \isomto \Hom_{\Q_\ell[G_K]}(V_\ell(A),V_
	\ell(B))
\]
for abelian varieties $A,B$ over $K$, where $V_\ell(-)$ denotes the rational Tate module. The Tate conjecture is also known for K3 surfaces by reduction to the case of abelian varieties via the Kuga--Satake construction \cite[Theorem 5.6(a)]{Tat:94a} (see also \cite{And96, Ben:15a}).

Recently Andreatta--Barbieri-Viale--Bertapelle \cite{ABVB16} have defined the filtered Ogus realisation for $1$-motives over $K$
\[
	T_{\FOg}:\mathcal{M}_{1,\Q} \to \FOg(K)
\]
where $\FOg(K)$ is the filtered Ogus category over $K$ (see \S\ref{sec:fog} for the definition) and $\mathcal{M}_{1,\QQ}$ is the category of $1$-motives up to isogeny (also called $1$-isomotives). Moreover  they proved that $T_{\FOg}$ is fully faithful. In particular for abelian varieties we have 
\[
	\Hom(A,B)\otimes \Q \cong \Hom_{\FOg(K)}(T_{\FOg}(A),T_{\FOg}(B))\ .
\]
The aim of this paper is to define a cohomology theory for $K$-varieties with values in $\FOg(K)$ compatible with $T_{\FOg}$. More precisely let $\mathbf{DM}_{gm}(K)$ be Voevodsky's category of geometric motives over $K$. Then we prove the following.
\begin{thr}
    There exists a (homological) realisation functor
   	\[
   		R_{\FOg}:\DM_{gm}(K)\to \D^b(\FOg(K))\ 
   	\]
	compatible with $T_{\FOg}$.
\end{thr}
 We use the approach of D\'eglise--Nizio\l{} \cite[Proposition 4.10]{DN18} to obtain the realisation. This in turn is based upon Nori's construction \cite{Huber2017} of an abelian category of mixed motives.  For a precise statement of the compatibility with $T_{\FOg}$ see \S\ref{ssec:1mot}. 
 
 As an illustration we can obtain a $\FOg$ version of the Tate conjecture for K3 surfaces over a number field. In fact the compatibility of de\thinspace Rham and crystalline cycle class maps \cite[Corollary 3.7]{BO83} gives rise to a homomorphism
\[
	c_{1,\FOg}:\Pic(X)\otimes \Q\rightarrow \Hom_{\FOg(K)}(K,H^2_\FOg(X)(1))
\]
and by using the full-faithfulness of \cite{ABVB16} in place of Faltings' theorem, we can similarly use the Kuga--Satake construction to show that the latter is surjective (hence an isomorphism) when $X$ is a K3 surface over $K$.

\subsection*{Acknowledgements} We thank Luca Barbieri-Viale because the present work has been inspired by one of his questions. We also thank Alessandra Bertapelle for useful discussion on the $\FOg$ category. Bruno Chiarellotto was supported by the Grant MIUR-PRIN  2015 ``Number  Theory  and  Arithmetic  Geometry''  and the Grant  PRAT 2015 ``Vanishing Cycles,  Irregularity and Ramification''. Christopher Lazda was supported by a Marie Curie fellowship of the Istituto Nazionale di Alta Matematica ``F. Severi'' and by the Netherlands Organization for Scientific Research (NWO).

\subsection{Notations and conventions}

Throughout this article, $K$ will denote a number field. A place of $K$ will always mean a finite place (we will never need to consider real or complex places). For every such place $v$ of $K$, let $K_v$ denote the completion, $\mathcal{O}_v$ the ring of integers, $k_v$ the residue field, $p_v$ its characteristic, and $q_v=p_v^{n_v}$ its order. For all $v$ which are unramified over $\Q$, let $\sigma_v$ denote the lift to $K_v$ of the absolute Frobenius of $k_v$. Following \cite{Huber2017}, a \emph{variety} over a field will be a reduced scheme, separated and quasi-projective over $k$; $\mathrm{Var}_K$ will denote the category of varieties over $K$, $\mathrm{Sm}_K$ the category of smooth varieties, and $\mathrm{Sm}_K^\mathrm{aff}$ the category of smooth affine varieties.

\section{The (filtered) Ogus category}\label{sec:fog}

We introduce the (filtered) Ogus category, following \cite{ABVB16}. Let $P$ be a cofinite set of absolutely unramified places of $K$. We define $\mathcal{C}_P$ to be the category whose objects are systems $M=(M_\dR,(M_v,\phi_v,\epsilon_v)_{v\in P})$ such that:
\begin{enumerate}
	\item $M_\dR$ is a finite dimensional $K$-vector space;
	\item $(M_v,\phi_v)$ is a $F$-$K_v$-isocrystal, that is, $M_v$ is equipped with a $\sigma_v$-linear automorphism $\phi_v$;
	\item $\epsilon=(\epsilon_v)_{v\in P}$ is a system of $K_v$-linear isomorphisms
	\[
		\epsilon_v:M_\dR\otimes K_v\to M_v\ .
	\]
\end{enumerate}

A morphism $f:M\to M'$ is then a collection $(f_\dR,(f_v)_{v\in P})$ where:
\begin{enumerate}
	\item $f_\dR:M_\dR\to M_\dR'$ is a $K$-linear map;
	\item $f_v:M_v\to M_v'$ is $K_v$-linear morphism compatible with Frobenius and such that $\epsilon_v^{-1} \circ f_v \circ \epsilon_v=f_\dR\otimes K_v$.
\end{enumerate}
Note that by the second criterion, to specify a morphism it is enough to specify $f_\dR$. There are obvious `forgetful' functors $\mathcal{C}_P\rightarrow \mathcal{C}_{P'}$ whenever $P'\subset P$ and we can form the Ogus category $\Og(K)$  as the 2-colimit
\[
	\Og(K)=2\colim_{P} \mathcal{C}_P
\]
where $P$ varies over all cofinite sets of unramified places of $K$. For an object $M\in \Og(K)$ and $n\in \Z$ we denote by $M(n)$ the Tate twist of $M$, that is where each Frobenius $\phi_v$ is multiplied by $p_v^{-n}$.

\begin{dfn} A \emph{weight filtration} on an object $M=(M_\dR,(M_v,\phi_v,\epsilon_v)_{v\in P})\in \mathcal{C}_P$ is an increasing filtration $W_\bullet M$ by subobjects in $\mathcal{C}_P$ such that for all $v\in P$ the graded pieces $\mathrm{Gr}_i^WM_v$ are pure of weight $i$. That is, all eigenvalues of the linear map $\phi_v^{n_v}$ are Weil numbers of $q_v$-weight $i$ (i.e. all their conjugates have absolute value $q_v^{i/2}$ \cite{Chi98}). Again, to give a weight filtration on $M$ it it suffices to give a filtration on $M_\dR$ which induces a weight filtration on all $M_v$.
\end{dfn}

We can therefore consider the filtered Ogus category $\FOg(K)$ whose objects are objects of $\Og(K)$ equipped with a weight filtration, and morphisms are required to be compatible with this filtration.

\begin{lmm}[\cite{ABVB16}, Lemma 1.3.2] The filtered Ogus category $\FOg(K)$ is a $\Q$-linear abelian category, and the forgetful functor
\[ \FOg(K)\rightarrow \Og(K) \]
is fully faithful. 
\end{lmm} 

\subsection{The basic construction}\label{ssec:basic} Let $X$ be a smooth variety over $K$. By Nagata plus Hironaka we can find a normal crossings compactification $X\subset \overline{X}$, with $D:=\overline{X}\setminus X$. Then we can consider the following cohomology groups
\begin{enumerate}
	\item (de Rham) There is an isomorphism
	\[ H_\dR^i(X/K) \cong H^i_{\log\text{-}\dR}((\overline{X},D)/K):=H_\dR^i(\overline{X},\Omega_{\overline{X}/K}\langle D\rangle), \]
where $\Omega_{\overline{X}/K}\langle D\rangle$ is the complex of algebraic differential forms on $\overline{X}$ with logarithmic poles along $D$. This is a finite dimensional $K$-vector space endowed with an increasing (weight) filtration $W_\bullet$, and a decreasing (Hodge) filtration $F^\bullet$. \cite[p.25]{Jan:95a}.
	\item (Rigid) There is a sufficiently divisible integer $n$ such that the pair $(\overline{X},D)$ has a model $(\overline{\mathcal{X}},\mathcal{D})$ over $\Spec(\mathcal{O}_K[1/n])$ such that $\overline{\mathcal{X}}/\mathcal{O}_K[1/n]$ is proper and smooth, and $\mathcal{D}\subset \overline{\mathcal{X}}$  a divisor with \emph{relative} normal crossings\footnote{In fact we can even suppose that the discriminant of $K$ divides  $n$.}.
	
	For any place $v$ of $K$ not dividing $n$, we can consider the rigid cohomology $H_{\rig}^i(\mathcal{X}_{k_v}/K_v)$  of the reduction modulo $v$ of $\mathcal{X}$, i.e.  $\mathcal{X}_{k_v}:=\mathcal{X} \otimes k_v$. Assuming that $v$ is not ramified in $K$, we can endow $H_{\rig}^i(\mathcal{X}_{k_v}/K_v)$ with a semilinear Frobenius endomorphism $\phi_v$. It turns out that $H_{\rig}^i(\mathcal{X}_{k_v}/K_v)$ is an $F$-isocrystal of mixed integral weights: all the eigenvalues of the linearised Frobenius $\phi_v^{n_v}$ are  Weil numbers of integral weight (relative to $k_v$) \cite{Chi98}.
	\item (Comparison) In the above setting the Berthelot (co-)specialisation map
	\[
		H_\dR^i(X/K) \otimes_K K_v \to H_{\rig}^i(\mathcal{X}_{k_v}/K_v)
	\] 
	is an isomorphism and it is functorial. By a result of Chiarellotto and Le Stum \cite{ChiLe:02a} we may identify (under the aforementioned isomorphism)
	\[
		\gr_s^W H_\dR^i(X/K) \otimes K_v = H_{\rig}^i(\mathcal{X}_{k_v}/K_v)^{\mathrm{wt}=s}
	\]
	where the latter is the sum of the all the generalised eigenspaces for $\phi_v^{n_v}$ associated to eigenvalues which are Weil numbers of $q_v$-weight $s$.
\end{enumerate}

\section{Realisation à la Nori}

A very general method for constructing realisations was given by Nori, and this was used to construct the derived syntomic realisation for varieties over $p$-adic fields in \cite{DN18}. For us, the basic point will be to construct appropriate $\FOg$-valued relative cohomology groups of a closed immersion $Y\hookrightarrow X$ of $K$-varieties; an appeal to Nori's basic lemma then allows the construction of $\FOg$-valued cohomology complexes which give rise to the required derived realisations. In order to construct the $\FOg$-structure on the relative cohomology $H^i(X,Y)$ we can follow \cite[Part~II, \S5.5]{PS08} and use descent to deduce the existence of a Frobenius compatible with the de\thinspace Rham weight filtration.

\subsection{Nori category}
The basic reference here is the exposition in \cite[Ch. II]{Huber2017} of Nori's original construction. Let $K$ be a field of characteristic $0$, and of cardinality $\leq \mathrm{card}(\C)$. We say that a system  $(X,Y,n)$ is a \emph{good pair} if: \begin{itemize}
\item $X$ is a $K$-variety;
\item $Y\subset X$ is a closed sub-variety;
\item for one (equivalently: for any) embedding $K\hookrightarrow \C$ the relative cohomology groups
\[H^i_B(X(\CC),Y(\CC),\QQ)\]
vanish for $i\neq n$.
\end{itemize}
Nori then considers a directed graph $\Delta_g^{\eff}$ whose vertices are exactly the set of good pairs over $K$, and which has the following two kinds of edges:
\begin{enumerate}
	\item (functoriality) $f^*:(X',Y',n)\to (X,Y,n)$  for any commutative square
	\[
		\xymatrix{
		X\ar[r]^f& X'\\
		Y\ar@{^{(}->}[u]\ar[r]_{f_{|Y}}&Y'\ar@{^{(}->}[u]
		}
	\]
	where $(X,Y,n),(X',Y',n)$ are good pairs.
	\item (coboundary) $\partial:(Y,Z,n-1)\to (X,Y,n)$ if $Z\subset Y\subset X$.
\end{enumerate}  
By definition, $\Delta_g$ is the directed graph obtained after localising $\Delta_g^{\eff}$ with respect to $(\GG_m,\{1\},1)$. The relative cohomology groups give a representation
\[
	H^*:\Delta_g\to \Mod_\QQ\qquad (\text{resp.}\ H^*:\Delta_g^{\eff}\to \Mod_\QQ)
\]
and the catgory of (cohomological) Nori motives $\CNM_K$ (resp. effective Nori motives $\CNM_K^\eff$) is the universal abelian category through which $H^*$ factors.
\subsection{The representability theorem of D\'eglise--Nizio\l} Here we briefly sketch a general method of D\'eglise and Nizio\l{} for constructing realisations, for more details the reader should consult \cite[\S4]{DN18}. Let $\mathscr{A}$ be a Tannakian $K$-linear category with a fibre functor $\omega:\mathscr{A}\to \mathrm{Vec}_\CC$ to the category of $\CC$-vector spaces. If there is representation $A:\Delta_g\to \mathscr{A}$ (i.e. a covariant functor) such that $\omega(A(X,Y,n))\cong H^n(X(\CC),Y(\CC),\CC)$, then there exists a motivic realisation (monoidal, covariant) functor
\[
	R:\DM_{gm}(K)\to D^b(\mathscr{A})
\]
such that $H^{-n}(RM(X,Y))=A(X,Y,n)^\vee$ for any good pair $(X,Y,n)$ (see \cite[Proposition 4.10]{DN18})\footnote{Mind that there is a misprint in loc. cit. relative to the cohomological  degree of LHS. They write $n$ instead of $-n$.}. Here we denote by $M(X,Y)$ the relative (homological) motive of the pair $(X,Y)$, characterised by the existence of an exact triangle
\[ M(Y)\rightarrow M(X) \rightarrow M(X,Y) \overset{+1}{\longrightarrow} \] 
in $\DM_{gm}(K)$. 
\begin{rmk}
	We note that the proof only uses affine schemes. So it is enough to work with affine good pairs. In fact, thanks to Beilinson \cite{Bei87}, any affine variety $X$ has a \emph{cellular stratification}
	\[
		F_\bullet X\ :\quad \varnothing=F_{-1}X\subset\cdots \subset F_dX=X
	\]
	such that: $(F_iX,F_{i-1}X,i)$ is a good pair; the complement $F_iX\setminus F_{i-1}X$ is smooth over $K$; either  $F_iX$ (resp. $F_{i-1}X$) is of dimension $i$ (resp. $i-1$), or $F_iX=F_{i-1}X$ is of dimension $<i$. 	Moreover, the set of cellular stratifications of a given $X$ form a filtered system, functorial in $X$. Thus for any affine scheme $X$ we can define the complex
	\[
		R'(X):=\colim_{F_\bullet}(A(F_0X,0)\to A(F_1X,F_{0}X,1)\to \cdots \to A(F_iX,F_{i-1}X,i)\cdots)
	\]
	which gives a functor $R':\left(\text{Sm}^\text{aff}_K\right)^\circ\to C^b(\text{Ind-}\mathscr{A})$, which is enough to construct $R$.
\end{rmk}

\section{Construction of \texorpdfstring{$\FOg$}{FOg}-valued cohomology}

In this section we will perform the key step in constructing the filtered Ogus realisation, by showing that relative de\thinspace Rham cohomology groups $H^i_\dR(X,Y/K)$ of $K$-varieties can be canonically enriched to the filtered Ogus category (Theorem \ref{theo: fogrel}). For smooth varieties, this follows from work of Chiarellotto and Le Stum \cite{ChiLe:02a}, and in general we use cohomological descent just as Peters and Steenbrink do in the mixed Hodge case in \cite[Part~II, \S~5.5]{PS08}.

\subsection{Cohomology of varieties with values in \texorpdfstring{$\FOg$}{FOg}}\label{sec: cohvar}

First, we will consider the case of a single variety $X/K$, and use cohomological descent to enrich the de\thinspace Rham cohomology groups of $X$ to the filtered Ogus category. 

\begin{dfn} \begin{enumerate}
\item A SNCD pair over $K$ will be a pair $(\overline{X},D)$ consisting of a smooth and proper $K$-variety together with a simple normal crossings divisor $D\subset \overline{X}$. A morphism $(\overline{Y},E)\rightarrow (\overline{X},D)$ of SNCD pairs is a morphism of varieties $f:\overline{Y}\rightarrow \overline{X}$ such that $f^{-1}(D)\subset E$, and the category of these objects will be denoted $\mathbf{SNC}_K$.
\item An SNCD resolution of a $K$-variety $X$ will be a simplicial SNCD pair $(\overline{X}_\bullet,D_\bullet)\in \mathbf{SNC}_K^{\Delta^\mathrm{op}}$, together with an augmentation $\pi_\bullet:X_\bullet:=\overline{X}_\bullet\setminus D_\bullet \rightarrow X$ which makes $X_\bullet$ a proper hypercover of $X$.
\item For a SNCD pair $(\overline{X},D)$ we denote
\[
	H^i_{\log\text{-}\dR}((\overline{X},D)/K):=H^i(\overline{X},\Omega_{\overline{X}/K}\langle D\rangle)
\]
its logarithmic de\thinspace Rham cohomology groups. Similarly for a simplicial SNCD pair.
\end{enumerate}
\end{dfn}

Let $(\overline{X}^{\leq n}_\bu,D_\bu^{\leq n})$ be an $n$-truncated SNCD pair over $K$. Then there exists a finite set of absolutely unramified primes $S\subset \norm{K}$ such that this $n$-truncated SNCD pair extends to an `$n$-truncated SNCD pair' 
\[ (\overline{\cur X}^{\leq n}_{\bu},\cur{D}^{\leq n}_{\bu})  \]
over the ring $\cur{O}_{K,S}$ of $S$-integers. In other words, $\overline{\cur X}^{\leq n}_{\bu}\rightarrow \spec{\cur{O}_{K,S}}$ is a smooth and proper $n$-truncated simplicial scheme, and $\cur{D}^{\leq n}_{\bu}\subset \overline{\cur{X}}^{\leq n}_{\bu}$ is a relative simple normal crossings divisor. We write $\mathcal{X}_\bu^{\leq n}:= \overline{\cur X}^{\leq n}_{\bu} \setminus \cur{D}^{\leq n}_{\bu}$. For any $v\not\in S$ we therefore obtain by \cite[Corollary 2.6]{BC94} an isomorphism
\begin{align*}
H^i_{\log\text{-}\dR}((\overline{X}^{\leq n}_\bu,D^{\leq n}_\bu)/K)\otimes_K K_v &\cong H^i_\rig( \mathcal{X}^{\leq n}_{\bu,k_v}/K_v) 
\end{align*} 
via which we can put a semilinear Frobenius endomorphism $\varphi_v$ on
\[ H^i_{\log\text{-}\dR}((\overline{X}^{\leq n}_\bu,D^{\leq n}_\bu)/K)\otimes_K K_v .\]
Since any two choices of model become isomorphic after possibly increasing $S$,we therefore obtain well-defined cohomology groups
\[ H^i_{\Og}(\overline{X}^{\leq n}_\bu,D^{\leq n}_\bu)\in \Og(K) \]
which are functorial in $(\overline{X}^{\leq n}_\bu,D^{\leq n}_\bu)$.

To show that these groups actually lie in the full sub-category $\FOg(K)\subset \Og(K)$, we need to produce a weight filtration. We consider the increasing weight filtration $W_\bu\Omega_{\overline{X}^{\leq n}_\bu}\langle D^{\leq n}_\bu \rangle$ on the logarithmic de\thinspace Rham complex of $(\overline{X}^{\leq n}_\bu,D^{\leq n}_\bu)$, that is the filtration coming from the number of log poles.

For each fixed $m\leq n$, we let $J_m$ denote the set of irreducible components of $D_m$; for any $I\subset J_m$ we write $D_{m,I}$ for the intersection of all elements of $I$, and $\norm{I}$ for the size of $I$. We let $a_I:D_{m,I}\rightarrow \overline{X}_m$ denote the natural closed immersion. Thus for any fixed $m$ we have by \cite[(3.1.5.2)]{Del:71b} that
\[ \gr^W_p\Omega_{\overline{X}_m}\langle D_m \rangle \cong \bigoplus_{I\subset J_m,\norm{I}=p} a_{I*} \Omega_{D_{m,I}}[-p]. \]
We therefore obtain a spectral sequence
	\begin{equation}
	\label{eq: ss1} \tag{$\star$} E^{p,q}_1 = \bigoplus_{i+j=p} \bigoplus_{\norm{I_j}=-i} H^{q+2i}_\dR(D_{j,I_j}/K) \Rightarrow H^{p+q}_{\log\text{-}\dR}((\overline{X}^{\leq n}_\bu,D^{\leq n}_\bu)/K)
\end{equation}
inducing a filtration $W_\bu$ on $H^{p+q}_{\log\text{-}\dR}((\overline{X}^{\leq n}_\bu,D^{\leq n}_\bu)/K)$.


\begin{prp} \label{prop: weights general} Let $(\overline{X}^{\leq n}_\bullet,D^{\leq n}_\bullet)$ be an $n$-truncated simplicial SNCD pair over $K$.
Then the filtration $W_\bu$ constructed above is a weight filtration, and exhibits $H^i_{\Og}(\overline{X}^{\leq n}_\bu,D^{\leq n}_\bu)$ as an object of the full subcategory $\FOg(K)\subset \Og(K)$.
\end{prp}
\begin{proof} Let $(\overline{\mathcal{X}}_\bu^{\leq n},\mathcal{D}_{\bu}^{\leq n})$ be a spreading out of $(\overline{X}^{\leq n}_\bullet,D^{\leq n}_\bullet)$ over some $\mathcal{O}_{K,S}$ as above, and set $\mathcal{X}_\bu^{\leq n}:= \overline{\cur X}^{\leq n}_{\bu} \setminus \cur{D}^{\leq n}_{\bu}$. Then using \emph{exactly} the same method as in \cite{ChiLe:02a} we can construct a similar spectral sequence
\[ \label{eq: ss2} \tag{$\star_{k_v}$} E^{p,q}_1 = \bigoplus_{i+j=p} \bigoplus_{\norm{I_j}=-i} H^{q+2i}_\rig(\mathcal{D}^{\leq n}_{j,I_j,k_v}/K_v) \Rightarrow H^{p+q}_{\rig}(\mathcal{X}_{\bu,k_v}^{\leq n}/K_v) \] 
abutting to the rigid cohomology of $\mathcal{X}_{k_v,\bu}$. Moreover, it again follows \emph{exactly} as in \cite{ChiLe:02a} that these two spectral sequences become isomorphic after tensoring (\ref{eq: ss1}) with $K_v$. Now, the spectral sequence (\ref{eq: ss2}) is \emph{not} compatible with Frobenius, however, it is so after making suitable Tate twists (essentially coming from the Gysin isomorphism). We therefore obtain a Frobenius compatible spectral sequence 
\[ E^{p,q}_1 = \bigoplus_{i+j=p} \bigoplus_{\norm{I_j}=-i} H^{q+2i}_\rig(\mathcal{D}^{\leq n}_{j,I_j,k_v}/K_v)(i) \Rightarrow H^{p+q}_{\rig}(\mathcal{X}_{\bu,k_v}^{\leq n}/K_v), \] 
and we now observe that each $E_1^{p,q}$ term is pure of weight $q$. Thus the induced filtration on
\[ H^{p+q}_{\rig}(\mathcal{X}^{\leq n}_{\bu,k_v}/K_v) \cong H^{p+q}_{\log\text{-}\dR}((\overline{X}^{\leq n}_\bu,D^{\leq n}_\bu)/K) \otimes K_v\]
is indeed a weight filtration for the action of Frobenius, as required.

\end{proof}

To get the analogous result  arbitrary $K$-varieties we appeal to  Nagata compactification and Hironaka's embedded resolution of singularities, which together imply that every $K$-variety $X$ admits an SNCD resolution. Moreover, by \cite[(5.3.5) II]{Del74} we know that if
\[ (\overline{X}_\bu,D_\bu),\;\;\pi_\bu: X_\bu\rightarrow X \]
is such an SNCD resolution, and $i<n-1$ then there are isomorphisms
\[  H^i_\dR(X/K) \isomto H^i_\dR(X^{\leq n}_\bu/K) \isomfrom H^i_{\log\text{-}\dR}((\overline{X}^{\leq n}_\bullet,D^{\leq n}_\bullet)/K). \]
By Proposition \ref{prop: weights general} we have a canonical enrichment of $H^i_{\log\text{-}\dR}((\overline{X}^{\leq n}_\bullet,D^{\leq n}_\bullet)/K)$ to $\FOg(K)$, which we can transport to  $H^i_\dR(X/K)$ via this isomorphism. To check that this structure doesn't depend on the choice of SNCD resolution, we argue along completely standard lines. That is, any two SNCD resolutions can be dominated by a third, and the pull-back maps induce isomorphisms on cohomology (for more details see \cite[\S8.2]{Del74}). We have therefore proved the following. 

\begin{crl} \label{cor: indep1} There is a canonical enrichment of the functor
\[ H^i_\dR(-/K):\mathrm{Var}_K^\mathrm{op}\longrightarrow \mathrm{Vec}_K  \]
to a functor
\[ H^i_{\FOg}(-): \mathrm{Var}_K^\mathrm{op}\longrightarrow \FOg(K) \]
taking values in the filtered Ogus category.
\end{crl}
\subsection{Cohomology of pairs with values in \texorpdfstring{$\FOg$}{FOg}}\label{sec: fogpairs}
Next we will deal with the relative cohomology of pairs. Suppose therefore that we are given a morphism
\[ 
f^{\leq n}_\bu:(\overline{Y}^{\leq n}_\bu,E^{\leq n}_\bu)\rightarrow (\overline{X}^{\leq n}_\bu,D^{\leq n}_\bullet)
 \]
of $n$-truncated simplicial SNCD pairs over $K$. As before we can spread out to obtain
\[ 
f^{\leq n}_\bu:(\overline{\mathcal{Y}}_\bu^{\leq n},\mathcal{E}_{\bu}^{\leq n}) \rightarrow(\overline{\mathcal{X}}_\bu^{\leq n},\mathcal{D}_{\bu}^{\leq n})
  \]
over some ring of integers $\mathcal{O}_{K,S}$, set $ \mathcal{X}^{\leq n}_\bu = \overline{\mathcal{X}}_\bu^{\leq n} \setminus \mathcal{D}_{\bu}^{\leq n}$ and $\mathcal{Y}^{\leq n}_\bu = \overline{\mathcal{Y}}_\bu^{\leq n} \setminus \mathcal{E}_{\bu}^{\leq n}$. For $v\not\in S$, the resulting comparison theorem
\begin{align*}
H^i_{\ldR}((\overline{X}^{\leq n}_\bu,D^{\leq n}_\bu),(\overline{Y}^{\leq n}_\bu,E^{\leq n}_\bu)/K)\otimes K_v &\cong H^i_\rig(\mathcal{X}^{\leq n}_{\bu,k_v},\mathcal{Y}^{\leq n}_{\bu,k_v}/K)
\end{align*} 
endows the LHS with a Frobenius structure, and thus gives rise to $\Og(K)$-valued cohomology groups $H^i_{\Og}((\overline{X}^{\leq n}_\bu,D_\bu),(\overline{Y}^{\leq n}_\bu,E_\bu))$. As for de Rham, the relative rigid cohomology is defined via mapping cone. To obtain a filtration we use \cite[Part I, Theorem 3.22]{PS08}  on the mapping cone of
\[ \Omega_{\overline{X}^{\leq n}_\bu}\langle D^{\leq n}_\bu \rangle \rightarrow  f^{\leq n}_{\bu*}\Omega_{\overline{Y}^{\leq n}_\bu}\langle  E^{\leq n}_\bu \rangle  \]
and We obtain a filtration $W_\bu$ on the cohomology groups 
\[ H^i_{\log\text{-}\dR}((\overline{X}^{\leq n}_\bu,D_\bu),(\overline{Y}^{\leq n}_\bu,E^{\leq n}_\bu)/K). \]

\begin{prp} Let $f^{\leq n}_\bu:(\overline{Y}^{\leq n}_\bu,E^{\leq n}_\bu)\rightarrow (\overline{X}^{\leq n}_\bu,D^{\leq n}_\bullet)$ be a morphism of $n$-truncated simplicial SNCD pairs over $K$. Then the filtration $W_\bu$ constructed above on $H^i_{\ldR}((\overline{X}^{\leq n}_\bu,D^{\leq n}_\bu),(\overline{Y}^{\leq n}_\bu,E^{\leq n}_\bu)/K)$ is indeed a weight filtration, and thus the cohomology groups $H^i_{\Og}((\overline{X}^{\leq n}_\bu,D^{\leq n}_\bu),(\overline{Y}^{\leq n}_\bu,E^{\leq n}_\bu)/K)$ lie in $\FOg(K)\subset \Og(K)$.
\end{prp}

\begin{proof}
We have a long exact sequence
\begin{align*}
\ldots \rightarrow H^i_{\ldR}((\overline{X}^{\leq n}_\bu,D^{\leq n}_\bu),&(\overline{Y}^{\leq n}_\bu,E^{\leq n}_\bu)/K) \rightarrow H^i_{\ldR}((\overline{X}^{\leq n}_\bu,D^{\leq n}_\bu)/K) \\
&\rightarrow H^i_{\ldR}((\overline{Y}^{\leq n}_\bu,E^{\leq n}_\bu)/K) \rightarrow \ldots
\end{align*}
which is the de Rham part of a long exact senquence in  $\Og(K)$. Moreover the groups $H^i_{\ldR}((\overline{X}^{\leq n}_\bu,D^{\leq n}_\bu)/K)$ and $H^i_{\ldR}((\overline{Y}^{\leq n}_\bu,E^{\leq n}_\bu)/K)$ underlie objects in $\FOg(K)$.  By mixed Hodge theory the above long exact sequence is \emph{strictly exact} with respect to the filtrations $W_\bu$.   If we therefore let $\ker^i$ denote the kernel of
\[ H^i_{\ldR}((\overline{X}^{\leq n}_\bu,D^{\leq n}_\bu)/K) \rightarrow H^i_{\ldR}((\overline{Y}^{\leq n}_\bu,E^{\leq n}_\bu)/K) \]
and $\mathrm{coker}^{i-1}$ the cokernel of 
\[ H^{i-1}_{\ldR}((\overline{X}^{\leq n}_\bu,D^{\leq n}_\bu)/K) \rightarrow H^{i-1}_{\ldR}((\overline{Y}^{\leq n}_\bu,E^{\leq n}_\bu)/K), \]
then the filtrations on $ H^i_{\ldR}((\overline{X}^{\leq n}_\bu,D^{\leq n}_\bu)/K)$ and $H^{i-1}_{\ldR}((\overline{Y}^{\leq n}_\bu,E^{\leq n}_\bu)/K)$ induce weight filtrations on $\ker^i$ and $\mathrm{coker}^{i-1}$ respectively, exhibiting them as objects of $\FOg(K)$. Thus for almost all unramified places $v$ of $K$, and for all integers $k$, we have an short exact sequence
\[ 0\rightarrow W_k\mathrm{coker}^{i-1}\rightarrow W_kH^i_{\ldR}((\overline{X}^{\leq n}_\bu,D^{\leq n}_\bu),(\overline{Y}^{\leq n}_\bu,E^{\leq n}_\bu)/K) \rightarrow W_k\ker^i\rightarrow 0 \]
of $K_v$-vector spaces. In particular, all Frobenius eigenvalues on the $k$th piece $W_kH^i_{\ldR}((\overline{X}^{\leq n}_\bu,D^{\leq n}_\bu),(\overline{Y}^{\leq n}_\bu,E^{\leq n}_\bu)/K) \otimes K_v$ are Weil numbers of weight $\leq k$, and thus $W_\bu$ is indeed a weight filtration on $H^i_{\ldR}((\overline{X}^{\leq n}_\bu,D^{\leq n}_\bu),(\overline{Y}^{\leq n}_\bu,E^{\leq n}_\bu)/K)$.
\end{proof}

Now if we are given a morphism $f:Y\rightarrow X$ of $K$-varieties, then we can always extend $f$ to a morphism of SNCD resolutions $(\overline{Y}_\bu,E_\bu)\rightarrow (\overline{X}_\bu,D_\bullet)$. Now arguing along essentially the same lines as in Proposition \ref{cor: indep1} we can show that the isomorphism
\[  H^i_\dR(X,Y/K)\cong H^i_{\ldR}((\overline{X}^{\leq n}_\bu,D^{\leq n}_\bu),(\overline{Y}^{\leq n}_\bu,E^{\leq n}_\bu)/K) \]
for $i<n-1$ allows us to view the former canonically as an object in $\FOg(K)$. If we let $\mathrm{Mor}_K$ denote the category of pairs of varieties over $K$, with morphisms just commutative diagrams, we therefore get the following result.

\begin{thr} \label{theo: fogrel} There is a canonical lifting \[ H^i_\FOg(-,-):\mathrm{Mor}_K \rightarrow \FOg(K)\]
of algebraic de\thinspace Rham cohomology
\[ H^i_\dR(-,-): \mathrm{Mor}_K \rightarrow \mathrm{Vec}_K \]
such that for any triple
\[ Z \rightarrow Y \rightarrow X \]
the long exact sequence in relative de\thinspace Rham cohomology induces a long exact sequence
\[ \ldots \rightarrow H^i_\FOg(X,Y) \rightarrow H^i_\FOg(X,Z)\rightarrow H^i_\FOg(Y,Z)\rightarrow \ldots \]
in $\FOg(K)$. In particular there is a realisation $H^*_\FOg:\CNM^\eff\to \FOg(K)$.
\end{thr}

Thus following the general method of \cite[\S4]{DN18} outlined above we can construct a (covariant) realisation functor
\[ R_\FOg: \mathbf{DM}_{gm}(K) \rightarrow D^b(\FOg(K)) \]
such that $H^{-n}(R_\FOg M(X)) \cong H^n_{\FOg}(X)^\vee$ for all $K$-varieties $X$.

\section{Compatibility with the realisation for \texorpdfstring{$1$}{1}-motives}\label{ssec:1mot}

In this section we want to compare the Ogus realization of $1$-motives \cite{ABVB16} with that for Nori motives. We follow \cite[\S~6.2]{AyoBar:15a}, but we will use a cohomological convention.

Let $S$ be a Noetherian scheme and $\pi:\overline{X}\to S$ a projective smooth scheme whose geometric fibers are connected curves of the same genus. The only cases we will use are $S=\spec{R}$ for $R= K,K_v,\mathcal{O}_v,k_v$. Then the fppf sheaf $\mathbf{R}^1\pi_*\mathbb{G}_{m, \overline{X}}$ is represented by a group scheme $\Pic_{\overline{X}/S}$ and the subfunctor $\Pic^0_{\overline{X}/S}$ of line bundles of degree zero on each fibre of $\pi$ is projective abelian scheme over $S$ \cite[Remark 5.26]{Kle:05a}. For any closed subscheme $i:Y\subset \overline{X}$, we have a surjective map
\[ \mathbb{G}_{m,\overline{X}} \rightarrow i_*\mathbb{G}_{m,Y}, \]
and we define the fppf sheaf $\mathbb{G}_{m, \overline{X}:Y}$ to be the kernel. If $Y$ is \'etale over $S$, then $\mathbf{R}^1\pi_*i_*\mathbb{G}_{m,Y}$ vanishes, and there is a short exact sequence of fppf sheaves
\[
	 0\to \pi_*i_*\mathbb{G}_{m, Y}/\pi_*\mathbb{G}_{m, \overline{X}}\to \mathbf{R}^1\pi_*\mathbb{G}_{m, \overline{X}:Y}\to\mathbf{R}^1\pi_*\mathbb{G}_{m, \overline{X}}\to 0\ .
\]
Thus $\mathbf{R}^1\pi_*\mathbb{G}_{m, \overline{X}:Y}$ is represented by an $S$-group scheme $\Pic_{\overline{X}:Y/S}$ which is an extension of $\Pic_{\overline{X}/S}$ by the $S$-torus $\pi_*i_*\mathbb{G}_{m, Y}/\pi_*\mathbb{G}_{m, \overline{X}}$ (cf. \cite[\S2.1]{BarSri:01a}). We let $\Pic_{\overline{X}:Y/S}^0$ denote the pullback of this extension to $\mathrm{Pic}^0_{\overline{X}/S}$, this is therefore a semi-abelian scheme over $S$.

Now let $Z\subset \overline{X}$ be another closed subscheme, \'etale over $S$ and such that $Y\cap Z=\varnothing$. We define $\Div_Z(\overline{X},Y)$ as the fppf sheaf associated to 
\[
	T/S\mapsto H^1_{Z_T}(\overline{X}_T,\mathbb{G}_{m,\overline{X}:Y}),
\]
By construction there is a natural map $u:\Div_Z(\overline{X},Y)\to \Pic_{\overline{X}:Y/S}$, and we can consider its pullback
\[
	u^0:\Div^0_Z(\overline{X},Y)\to \Pic^0_{\overline{X}:Y/S}
\]
to $\Pic^0_{\overline{X}/S}$. This object is a $1$-motive over $S$, and we denote it by $\Pic^+(X,Y)$ (or $\Pic^+(X)$ when $Y=\varnothing)$ where $X:=\overline{X}\setminus Z$. This is the version over $S$ of the motive defined in \cite[Def. 2.2.1]{BarSri:01a}.
 
Let $\Delta_1^{\eff}\subset \Delta_g^{\eff}$ be the full sub-diagram whose vertices are $(X,Y,1)$ for $X$ a smooth affine curve over $K$, $Y$ a closed subset consisting of finitely many closed points of $X$. We denote by $\CNM^\eff_1$ the Nori category universal\footnote{This is the cohomolgical version of the category $\mathsf{EHM}_1''$ of \cite{AyoBar:15a}.} for the standard representation
\begin{align*}
	H^1:\Delta_1^{\eff} &\to \Mod_\QQ\qquad \\ (X,Y,1)&\mapsto H^1(X(\CC),Y(\CC),\QQ).
\end{align*}
Moreover  we can define the following representation
\[
	\Pic^+:\Delta_1^{\eff} \to \mathcal{M}_{1} \qquad (X,Y,1)\mapsto \Pic^+(X,Y)=[\Div_Z^0(\overline{X},Y)\to \Pic^0_{\overline{X}: Y/K}]
\]
where $\overline{X}$ is the smooth compactification of $X$ and $Z=\overline{X}\setminus X$ is the boundary divisor. By universality this functor factors through $\CNM^\eff_1$ and it is show in \cite[Theorem 5.6]{AyoBar:15a} that this induces equivalence of categories
\[
	\CNM^\eff_1\xrightarrow{\cong} \mathcal{M}_{1,\QQ} \ . 
\]
\begin{prp}
	There is a functorial isomorphism
	\[ T_{\FOg}\left(\Pic^+ (X,Y) \right)= H^1_\FOg(X,Y)(1).\]
	In particular the filtered Ogus realisation is compatible with that on one motives, in the sense that the diagram
	\[  \xymatrix{  \mathcal{M}_{1,\QQ} \ar@{^(->}[d] \ar[rr]^-{T_{\FOg}} & & \FOg(K)  \ar@{^(->}[d] \\ \mathbf{DM}_{gm}(K) \ar[rr]^-{R_\FOg} & &  D^b(\FOg(K))    } \]
	commutes up to natural isomorphism.
\end{prp}
\begin{proof}
	In \cite[Proposition~8.3]{AyoBar:15a} the authors show the (homological) compatibility between the embedding of 1-motives in $\DM^{\eff}_{gm}$ and that in (the homological version of) $D^b(\CNM)$. The cohomological version is just a reformulation; the second claim therefore follows from the first.
	
Let $T_\dR$ denote the de\thinspace Rham realisation on $\mathcal{M}_1$. To prove the first, we note that \cite[Lemma 2.6.2]{BarSri:01a} provides isomorphisms
\[
	T_\dR(\Pic^+(X,Y)) \overset{\cong}{\longrightarrow}H^1(\overline{X}, \calO_{\overline{X}}(-Y)\to \Omega^1_{\overline{X}}\langle Y+Z \rangle(-Y) ) \overset{\cong}{\longrightarrow} H^1_\dR(X,Y/K)
\]
of $K$-vector spaces. Indeed, the de Rham cohomology of $X$ is computed by the complex $\calO_{\overline{X}}\to \Omega^1_{\overline{X}}\langle Z \rangle$ and that of $Y$ by $\calO_Y$. Thus the relative cohomology is computed by the complex $\calO_{\overline{X}}(-Y)\to \Omega^1_{\overline{X}}\langle Z \rangle$ and it is enough to note that $ \Omega^1_{\overline{X}}\langle Z \rangle= \Omega^1_{\overline{X}}\langle Y+Z \rangle(-Y)$ to conclude.

Now let $v$ is an unramified place of good reduction for the triple $(\overline{X},X,Y)$; in other words not only does $\overline{X}$ have good reduction $\overline{\mathcal{X}}\rightarrow \spec{\mathcal{O}_v}$, but the complementary divisors $Y$ and $Z=\overline{X} \setminus X$ extend to disjoint closed subschemes $\mathcal{Y},\mathcal{Z} \subset \overline{\mathcal{X}}$ which are \'etale over $\mathcal{O}_v$. We therefore have a 1-motive 
\[ \Pic^+(\mathcal{X},\mathcal{Y}) \]
over $\mathcal{O}_v$ exhiniting the good reduction of $\Pic^+(X,Y)$, and by \cite[Corollary 4.2.1]{AndBar:05a} we obtain an isomorphism
\[ T_\dR\left(\Pic^+(X,Y)\right) \otimes_K K_v \overset{\cong}{\longrightarrow}  T_\mathrm{cris}\left(\Pic^+(\mathcal{X},\mathcal{Y})_{k_v}\right) \]
of $K_v$-vector spaces, where $T_\mathrm{cris}$ is the crystalline realisation of $\mathcal{M}_1$ over $k_v$. Now, we have an isomorphism of vector spaces
\[ H^1_\dR(X_{K_v},Y_{K_v}/K_v)\to H^1_\rig(\mathcal{X}_{k_v},\mathcal{Y}_{k_v}/K_v)\]
which concretely is induced by the map
\[
	[I\calO_{X_{K_v}^\an}\to \Omega^1_{X_{K_v}^\an}]\to [Ij^\dag\calO_{X_{K_v}^\an}\to j^\dag\Omega^1_{X_{K_v}^\an}]
\]
of complexes over $X_{K_v}^\an$, where $I$ is the ideal of $Y_{K_v}^\an$ (it easy to check that $Ij^\dag\calO_{X_{K_v}^\an}\to j^\dag\Omega^1_{X_{K_v}^\an}$ computes the relative rigid cohomology $H^1_\rig(X_{k_v},Y_{k_v}/K_v)$). This induces an isomorphism
\[ T_\mathrm{cris}\left(\Pic^+(\mathcal{X},\mathcal{Y})_{k_v}\right) \rightarrow H^1_\rig(\mathcal{X}_{k_v},\mathcal{Y}_{k_v}/K_v) \]
of $K_v$-vector spaces, and to conclude we need to show that this induces a Frobenius invariant isomorphism 
\[ T_\mathrm{cris}\left(\Pic^+(\mathcal{X},\mathcal{Y})_{k_v}\right) \rightarrow H^1_\rig(\mathcal{X}_{k_v},\mathcal{Y}_{k_v}/K_v)(1). \]
The key observation now is that in fact we can argue by d\'evissage on weights. Indeed, we have commutative diagrams
\[ \xymatrix{  0 \ar[r] & T_\mathrm{cris}\left(\ker c\right) \ar[r]\ar[d]^{\cong} & T_\mathrm{cris} \left(\Pic^+(\mathcal{X},\mathcal{Y})_{k_v}\right) \ar[r]^c\ar[d]^{\cong} & \ar[r]\ar[d]^{\cong} T_\mathrm{cris} \left(\Pic^+(\mathcal{X})_{k_v}\right)  & 0 \\
0 \ar[r] & \frac{H^0_\rig(\mathcal{Y}_{k_v})(1)}{H^0_\rig(\mathcal{X}_{k_v})(1)}\ar[r] & H^1_\rig(\mathcal{X}_{k_v},\mathcal{Y}_{k_v}/K_v)(1) \ar[r] & H^1_\rig(\mathcal{X}_{k_v}/K_v)(1) \ar[r] & 0  } \]
and 
\[ \xymatrix@C=2.2em{  0 \ar[r] & T_\mathrm{cris}\left( \Pic^+(\overline{\mathcal{X}})_{k_v}\right) \ar[r]\ar[d]^{\cong} & T_\mathrm{cris} \left(\Pic^+(\mathcal{X})_{k_v}\right) \ar[r]\ar[d]^{\cong} & \ar[r]\ar[d]^{\cong} T_\mathrm{cris} \left(\Div_{\mathcal{Z}}^0(\overline{\mathcal{X}})_{k_v} \right)  & 0 \\
0 \ar[r] & H^1_\rig(\overline{\mathcal{X}}_{k_v}/K_v)(1) \ar[r] & H^1_\rig(\mathcal{X}_{k_v}/K_v)(1) \ar[r] & H^0_\rig(\mathcal{Z}_{k_v}/K_v) \ar[r] & 0  } \]
with all rows exact. Since the  pieces
\begin{align*} T_\mathrm{cris}\left(\ker c \right) &\overset{\cong}{\longrightarrow} \frac{H^0_\rig(\mathcal{Y}_{k_v}/K_v)(1)}{H^0_\rig(\mathcal{X}_{k_v}/K_v)(1)} \\ 
T_\mathrm{cris}\left( \Pic^+(\overline{\mathcal{X}})_{k_v}\right) &\overset{\cong}{\longrightarrow} H^1_\rig(\overline{\mathcal{X}}_{k_v}/K_v)(1)   \\
T_\mathrm{cris} \left(\Div_{\mathcal{Z}}^0(\overline{\mathcal{X}})_{k_v} \right) &\overset{\cong} {\longrightarrow} H^0_\rig(\mathcal{Z}_{k_v}/K_v)
\end{align*}
are pure of weights $-2,-1$ and $0$ respectively, we can use the fact that the weight filtration on an $F$-isocrystal over $K_v$ is \emph{canonically} split to show that it suffices to verify the Frobenius compatibility on these pure graded pieces. The only non trivial Frobenius shows up in weight $-1$ where the comparison is proved by Andreatta and Barbieri-Viale \cite[Theorem B']{AndBar:05a} since $H^1_\rig(\overline{\mathcal{X}}_{k_v}/K_v) \cong H^1_\mathrm{cris}(\overline{\mathcal{X}}_{k_v}/\mathcal{O}_v)[1/p]$. 
\end{proof}

\section{The \texorpdfstring{$\FOg$}{FOg} avatar of the Tate conjecture}
Let $K$ be a number field and $X$ be a smooth and projective variety over $K$. Fix a finite set of unramified places $S$ such that $X$ extends to a smooth and proper scheme $\mathcal{X}\rightarrow \mathcal{O}_{K,S}$. We may consider the de\thinspace Rham cycle class map
\[ c_{1,\dR}:\mathrm{Pic}(X) \rightarrow H^2_\dR(X/K) \]
and, for any place $v\not\in S$, the crystalline cycle class map
\[ c_{1,\mathrm{cris}}:\mathrm{Pic}(X) \rightarrow \mathrm{Pic}(X_{K_v}) \overset{\cong}{\longleftarrow} \mathrm{Pic}(\mathcal{X}_{\mathcal{O}_v}) \rightarrow \mathrm{Pic}(\mathcal{X}_{k_v}) \rightarrow H^2_\mathrm{cris}(\mathcal{X}_{k_v}/\mathcal{O}_v)[1/p], \]
which by \cite[Corollary 3.7]{BO83} are compatible via the comparison isomorphism\footnote{The compatiblity of the crystalline and de Rham cycle class has been generalized to the rigid setting in \cite{CCM13,DegMaz:15a}.}
\[ H^2_\dR(X/K)\otimes K_v \isomto H^2_\mathrm{cris}(\mathcal{X}_{k_v}/\mathcal{O}_v)[1/p]. \] 
Since the image of $c_{1,\mathrm{cris}}$ is contained within the subspace of $H^2_\mathrm{cris}(\mathcal{X}_{k_v}/K_v)$ on which Frobenius $\phi_v$ acts via multiplication by $p_v$, we obtain an induced cycle class map
\[ c_{1,\FOg}:\mathrm{Pic}(X) \rightarrow \mathrm{Hom}_{\FOg(K)}(\mathbbm{1},H^2_{\FOg}(X)(1))=: \mathcal{T}^1(X). \]
where $\mathbbm{1}:=(K,(K_v,\sigma_v))$ is the unit object of $\FOg(K)$. Following the argument outlined in \cite[(5.6)]{Tat:94a}, we can show that whenever $X$ is a K3 surface, an `Ogus' version of the Tate conjecture holds for $X$, describing the \emph{rational} Picard group $\mathrm{Pic}(X)\otimes \QQ$.

\begin{thr}
	Let $X/K$ be a K3 surface Then the cycle class map
	\[
		c_{1,\FOg} :\Pic(X)\otimes \QQ\to \mathcal{T}^1(X)
	\]
	is surjective, and therefore an isomorphism.
\end{thr}
\begin{proof}
	 To save notation, we will write $[\mathcal{L}]=c_{1,\FOg}(\mathcal{L}) \in H^2_\FOg(X)(1)$ for any line bundle on $X$. First of all, we can show that it is enough to prove the above theorem up to finite base change. Indeed, if we let $F/K$ be a finite extension and $\alpha\in \mathcal{T}^1(X)$ then, assuming the result holds for $X_F$, we can write
	$\alpha = \sum_i \lambda_i [\mathcal{L}_i]$, for some $\mathcal{L}_i\in \Pic(X_F)$ and $\lambda_i\in \QQ$. Now we can apply the norm to obtain
	\[
		\alpha = \frac{1}{[F:K]}\sum_i \lambda_i [N_{F/L}(\mathcal{L}_i)]
	\]
as required.

Now, fix an ample line bundle $\mathcal{L}$ on $X$, and let $P^2_{\FOg}(X)(1)\subset H^2_\FOg(X)(1)$ denote the subspace of primitive classes, that is the orthogonal complement to $K\cdot [\mathcal{L}]$ under the perfect pairing
\[ H^2_\FOg(X)(1) \otimes H^2_\FOg(X)(1) \rightarrow \mathbbm{1}. \]
Then we have a direct sum decomposition $H^2_\FOg(X)(1)\cong P^2_\FOg(X)(1)\oplus \mathbbm{1}$ in $\FOg(K)$, which implies that $\mathcal{T}^1(X)=(P^2_\FOg(X)(1) \cap \mathcal{T}^1(X))\oplus  (\mathbbm{1}\cap \mathcal{T}^1(X)$). Since  $\mathbbm{1}\cap \mathcal{T}^1(X)=\QQ\cdot [\mathcal{L}]$  we only need to prove the statement for  $\alpha\in P^2_\FOg(X)(1) \cap \mathcal{T}^1(X)$.

	Let 
	\[
		j_\FOg:
	P^2_\FOg(X)(1) \to \underline{\End}_\FOg(H^1_{\FOg}(A))\ 
	\]
	be the inclusion induced by the Kuga--Satake construction  \cite[Theorem 7.3]{Ogu84} - since it suffices to prove the claim after a finite extension we can assume that everything is defined over $K$. By full faithfulness of the $\FOg$-realisation on abelian varieties \cite[Theorem 3.14]{ABVB16} we have $j_\FOg(\alpha)=\sum_i\lambda_i[f_i]_\FOg$, for $f_i\in \End(A)$ and $\lambda_i\in \QQ$.  
	
	Now embed $K$ into $\C$ and consider the analogous picture in Betti cohomology (subscript $B$ stands for Betti cohomology)
	\[
		j_B:P^2_B(X(\CC),\QQ(1))\hookrightarrow \End_{\QQ}(H^1_B(A(\CC),\QQ))
	\] 
	which admits a retraction $q_B$ as the target is a polarised pure Hodge structure and the category of polarised pure Hodge structure is semi-simple \cite[Corollary 2.12]{PS08}. Thus $q_B([f_i]_B)$ is a Hodge class since $[f_i]_B$ is so. It follows by the Lefschetz (1,1)-theorem that there exists a line bundle $\mathcal{L}_i\in \Pic(X_\CC)$ such that $[\mathcal{L}_i]_B=q_B([f_i]_B)$. After replacing $K$ by a finite extension, we can assume that all the $\mathcal{L}_i$ are defined over $K$. By the compatibility of the Betti and de\thinspace Rham cycle class maps we get the equality
	\[
		\alpha = q_\dR j_\dR(\alpha)=\sum_i \lambda_i q_\dR [f_i]_\dR =\sum_i \lambda_i [\mathcal{L}_i]_\dR
	\]
inside $P^2_\dR(X_{\C})(1)$, hence we find $\alpha = \sum_i \lambda_i [\mathcal{L}_i]_\dR $ inside $P^2_\dR(X)(1)$, and the proof is complete.
\end{proof}

\bibliographystyle{plain}

\end{document}